\documentclass{amsart}

\usepackage{latexsym}
\usepackage{amsmath,amscd,latexsym}
\usepackage[psamsfonts]{amssymb}

\usepackage{amsmath,enumerate}
\usepackage[psamsfonts]{amssymb}
\usepackage{amsthm}
\usepackage[mathscr]{eucal}

\def\H{\mathcal{H}}

\usepackage{graphicx}

\usepackage[ansinew]{inputenc}
\usepackage{colortbl}
\usepackage{xcolor}
\usepackage[active]{srcltx}

\newtheorem{theorem}{Theorem}[section]
\newtheorem{lemma}[theorem]{Lemma}

\newtheorem{remark}[theorem]{Remark}
\newtheorem{definition}[theorem]{Definition}

\newcommand{\R}{\mathbb{R}}

\def\RR{\mathbb{R}}

\def\bx{{\bf x}}

\renewcommand{\H}{{\mathcal H}}

\def\div{\mbox{{\rm div}}}

\def\va{\raise 2pt\hbox{,}}

\title[Porous-media flux--saturated Keller--Segel models]
      {Cross--diffusion and traveling waves in porous-media flux--saturated Keller--Segel models}

\author[M. Arias, J. Campos \& J. Soler]{Margarita Arias, Juan Campos and Juan Soler}

\address{Departamento de Matem\'atica Aplicada and Excellence Research Unit ``Modeling Nature'' (MNat), Facultad de Ciencias. Universidad de Granada. 10871-Granada. Spain }
\email{marias@ugr.es, campos@ugr.es, jsoler@ugr.es}
\subjclass{Primary: 35K57, 35A15, 35C07}
\keywords{Flux-saturated, Keller-Segel, Blow-up, Traveling waves, Cross-diffusion, Solitons}
\thanks{ This work was partially supported in part by MINECO (Spain), 
 project  MTM2014-53406-R, FEDER resources, and Junta de Andaluc\'{\i}a Project  P12-FQM-954.}
\begin{document}


\begin{abstract}

This paper  deals with the analysis of qualitative properties involved in the dynamics of Keller--Segel type systems in which the diffusion mechanisms of the cells are driven by porous-media flux--saturated  phenomena.
We study the regularization inside the support of a solution with jump discontinuity  at the boundary of the support. We analyze the behavior of the size of the support and blow--up of the solution,  and the possible convergence in finite time towards a Dirac mass in terms of the three constants of the system: the mass, the flux--saturated characteristic speed, and the chemoattractant sensitivity constant. These constants of motion also characterize the dynamics of regular and singular traveling waves.
\end{abstract}
\maketitle

\section{Introduction}
Cross-diffusion is the phenomenon in which a gradient in the concentration of one species affects the  diffusion flux of another species. The aim of this paper is to study spatiotemporal patterns that occur in cell aggregation induced by the attracting chemical signal produced by the cells themselves leading to a reacting system which is  counterbalanced by cell diffusion with a saturated flux. Patterns  arising by signals exerting a concentration-dependent mechanism constitute a  class of self--organized structures, which are motors of generation of cellular diversity like Turing structures or standing waves. In the classical Keller--Segel model, the chemoattractant is emitted by the cells that react according to biased random walk inducing linear diffusion operators.
Chemotaxis can be defined as the movement or orientation of a population (bacteria, cell or other single or multicellular organisms) induced by a chemical concentration gradient either towards or away from the chemical signals. The original Keller--Segel model consists in a reaction--diffusion system of
two coupled parabolic equations:
\begin{equation} 
\left\{
\begin{array}{l}
 \partial_t u  =
\div_{\bx} (D_u \nabla_\bx u - \chi u \nabla_\bx S)+H(u,S),
{} \\
{} \\
 \alpha \partial_t S   = D_S \Delta S + K(u,S), \label{KS}
\end{array} \right.
\end{equation}
where $u=u(t,\bx)$ is the cell density at position $\bx$ and time
$t$, and $S=S(t,\bx)$ is the density of the chemoattractant. The
positive definite terms $D_S$ and $D_u$ are the diffusivity of the
chemoattractant and of the cells, respectively,  $\chi\geq
0$ is the chemotactic sensitivity, $\alpha>0$ represents a time scale of the process, and the functions $H(u,S)$ and $K(u,S)$
in (\ref{KS}) model the interaction (production and degradation) between the cell
density and the chemical substance. In most simplified models and in the original Keller--Segel system, these terms are modeled as $K(u,S)=u -S$ and $H(u,S)=0$.

Self--organized patterning in reaction--diffusion systems driven by linear diffusion has been extensively studied since the seminal paper of Turing about morphogenesis. The interaction between diffusion and reaction has shown to yield rich and unexpected phenomena in different contexts. Theoretical and mathematical modeling of chemotaxis
starts with the works of Patlak in the 1950s and Keller and Segel in the 1970s \cite{[KS71],[K80],[PA]}.
Depending upon the size of cell density and the spatial dimensions under consideration, the nonlocal chemical interaction could dominate diffusion and/or produce blow--up of the solutions \cite{DP,[DS09]}, where we refer to blow--up according to the classic concept of loss of $L^\infty$ estimate of the solution  
 $u$. For
the two--dimensional case, in \cite{B,DP} was proved that, if the mass $M$
exceeds $\frac 8 \pi$, then any classical solution
blows up in finite time, while if 
$M < \frac 8 \pi$, then a classical solution exists globally,
since the diffusion dominates the aggregation.
However, in one dimension of space diffusion is always stronger than aggregation and blow--up never occurs for
 system (\ref{KS}), see \cite{H}. The situation could be modified in one space dimension
by changing the linear diffusion approach in the original Keller--Segel model. For example, when cell diffusion is ruled by fractional
diffusion, blow--up may or may not occur depending on the initial data \cite{BC}, see also \cite{[CH2]}. The use of the Keller--Segel model as a prototype in chemotaxis phenomena is motivated by  the simplicity of the associated reaction--diffusion equations, and  its ability to capture 
 nature experiments as well as the capacity to predict the dynamics of the collective motion of both species: cells and chemoattractant. In fact, this model has been extensively studied in a large variety of spatial patterns for the E. coli, and has also emerged as an important tool in the
understanding of embryonic pattern forming 
processes, such as the formation of spots or pigmentation patterning 
in shells or in the skin of several animals \cite{KM}, the movement of sperm towards the egg during
fertilization, migration of lymphocytes, prediction of the tumor cell-induced angiogenesis, and 
macrophage invasion into tumor, among others, see \cite{[AQ08],[FK02],[HW00],[HP],KM} for some applications in different contexts.

In the context of morphogenesis, linear diffusion implies that cells receive information instantaneously and passively and has so far posed irresolvable problems from a modeling point of view since it predicts the same exposure time for each cell in the whole tissue, which is not realistic. However,  as it has been pointed out in the experiments, see \cite{[DYH07]}, the concentration of morphogen received by the cells and the time of exposure are of similar relevance in order that real  biological patterns could be developed, because morphogen signals have cumulative effects. The used alternative of arbitrarily cutting off the tail of the Gaussian morphogen distribution induced by linear diffusion models contributes to create artificial fronts \cite{VGRaS},  while the real propagation fronts are far from those predicted by linear diffusion models.
 
A possible alternative to  linear Fickian diffusion  is to consider flux--saturated mechanisms. A prototype of the kind of degenerate nonlinear diffusion  phenomena to be considered in this paper is the following 
\begin{eqnarray}
\label{parab}
\displaystyle \frac{\partial u}{\partial t}= \nu \, \frac{\partial}{\partial x}\left( \frac{|u|  \frac{\partial u^m}{\partial x}}{\sqrt{1+\frac{\nu^2}{c^2} \left|\frac{\partial u^m}{\partial x}\right|^2}} \right) + \mbox{ reaction terms},\quad  \nu, c>0,\, m\ge 1,
\end{eqnarray}
which combines flux--saturation effects together with those of porous media flow. It will be also considered the so called relativistic heat equation
\begin{eqnarray}
\label{parab-rel}
\displaystyle \frac{\partial u}{\partial t}= \nu \, \frac{\partial}{\partial x}\left( \frac{|u|  \frac{\partial u}{\partial x}}{\sqrt{u^2+\frac{\nu^2}{c^2} \left|\frac{\partial u}{\partial x}\right|^2}} \right) + \mbox{ reaction terms},\quad  \nu, c>0.
\end{eqnarray}
The derivation of these models using optimal transport theory, the fact that the propagation speed of discontinuous interfaces is generically bonded by $c$, and the convergence to the classical porous medium equation  have been studied in \cite{Campos2016,preprintVicent,CCC2}. The  Rankine--Hugoniot characterization of traveling jump discontinuities, traveling waves, waiting time and
smoothing effects have been previously treated in \cite{CCCSS-EMS}. However, the proposed Keller--Segel  model with flux--saturated mechanism will require of a more precise analysis.  Other  flux--saturated porous--media systems have been analyzed in \cite{CCCSS-SUR,CCCSS-INV}.
Flux--saturated equations were introduced in \cite{CKR,Ros-Non,Levermore81,MM,Rosenau2}, while the general theory for the existence of entropy solutions associated 
with flux--saturated equations has been widely developed in the framework of Bounded Variation functions, see \cite{ACMSV,ARMAsupp,ACM2005,CMSV,Giacomelli}. Regularity properties were analyzed in \cite{ACMSV,CCCSS-SUR,CCCSS-EMS,Carrillo}. Applications of these ideas to diverse contexts such as Physics, Astronomy or Biology can  be found for instance in \cite{CCCSS-SUR,Levermore81,Rosenau2,VGRaS}. The deduction of the Keller--Segel equations from the foundations of the kinetic theory as a combination of parabolic and hyperbolic limits is also of interest, see \cite{[BBNS07],[BBNS10],[CMPS04],[CDMOSS06],[LAC],[OH02]} for the derivation of the classical Keller--Segel system and \cite{[BBNS10B],BBTW,[CH2]} in the case of flux--saturated or porous media mechanisms of spreading.

Another idea in the line of adjusting the model to experimental evidence is to saturate the chemotactic flux function by modifying the dependence of the cell velocity term, which is proportional to the magnitude of the chemoattractant gradient $\chi u \nabla_\bx S$ in the original Keller--Segel model  \eqref{KS}. This could be done by optimizing the flux density of cells  along the trajectory induced by the chemoattractant, giving  rise to terms
 of type
\begin{eqnarray}
\div_{\bx}\left(\chi u \frac{\nabla_\bx S}{\sqrt{1 + |\nabla_\bx
S|^2}} \right),
\label{optimal}
\end{eqnarray}
in the corresponding Euler-Lagrange equation, which is a mean curvature type operator. This was first introduced in \cite{[BBNS10B]} and then numerically analyzed in \cite{cher}. 
Several authors have treated the Keller--Segel problem with some type of saturated diffusion since the pioneering work  \cite{[BBNS10B]}, see for instance \cite{BBTW,BW2,BW1,cher}. They have discussed the problem of well--posedness and the possibility of blow--up in the case that the cells motion have a saturated flux given by the relativistic heat equation \eqref{parab-rel} with different boundary conditions.

On the other hand, we are interested in exploring some particular solutions associated with the flux--saturated Keller--Segel model. Traveling waves, i. e. solutions of the type $u(t,x) = u(x - \sigma t)$, are archetypal models for pattern formation.   In the linear diffusion case, the  argument justifying that even if the solution has not compact support,  its size (invoking mass or concentration according to the case dealt with) is very small out of some ball with large radius, might be unrealistic, as we have pointed out before \cite{[DYH07],VGRaS}. Then, exploring or modeling new nonlinear transport/diffusion phenomena is an interesting subject not only from the viewpoint of applications, but also from a mathematical perspective.  In the framework of flux--saturated mechanisms, traveling waves have been studied in \cite{CCCSS-INV,CCCSS-SUR,Campos,Campos2016}. In this context, as we will show, the Keller--Segel system with flux--saturated mechanism (\ref{s}), which is the object of study in this paper,  exhibits new
properties with respect to the classical one.
Indeed, the existence of singular traveling waves is one of these new properties, which are of solitonic type if the corresponding initial data have compact support with finite mass.

Then, this paper deals with the problem of finding functions $u$ and $S$ verifying
\begin{equation}\label{s}
\left\{\begin{array}{ll}
\begin{array}{l}
\displaystyle {\frac{\partial u}{\partial t}}= \frac{\partial}{\partial x} \left(u\Phi
\left(u^{m-1}{\frac{\partial u}{\partial x}}\right) -a{\frac{\partial S}{\partial x}}u \right),\\ \\
\displaystyle {-\frac{\partial^2 S}{\partial x^2}}=  u,\end{array}& \quad x\in \R, \; t>0,
\end{array}\right.
\end{equation}
being $a$ a positive constant, $m\geq 0$ and $\Phi$ such that the following hypotheses are satisied:

\begin{equation}\tag{$H_1$}\label{H1}
\Phi \in C^2(\R), \quad \Phi(-y)=-\Phi(y), \quad  \Phi '(y)>0, \quad \lim_{y\to +\infty}\Phi(y)=c>0,
\end{equation}
for every $y \in \R.$ Moreover,
\begin{equation}\tag{$H_2$}\label{H2}
  \Phi '(y)=O(y^{-\alpha-1}),
   \mbox{ with } \alpha\geq 2.
  \end{equation}

Therefore, $c-\Phi(y)=O(y^{-\alpha})$. Note that from the definition of $\Phi$, the relativistic heat equation \eqref{parab-rel} is included for $m=0$ and $\alpha=2$, and the flux--saturated porous media equation \eqref{parab} is also considered for $m>0$. We want also to remark that modifying \eqref{s} by saturating the chemotactic flux  in the way \eqref{optimal} could be also possible in our analysis, but we omit here for simplicity.

\begin{remark}\label{R1}
Setting $g:=\Phi^{-1}$, as a consequence of the above assumptions we obviously have  $g(r)=O((c-r)^{-\frac{1}{\alpha}})$  and $g'(r)=O((c-r)^{-\frac{\alpha +1}{\alpha}})$. In particular,
 since $0<\frac{1}{\alpha}\leq\frac 1 2$, defining
$$G(u):=\int_0^ug(\sigma)\, d\sigma, \; u\in [-c,c], $$
we find that $G\in C[-c,c]\cap C^3(-c,c)$ and verifies $G(-u)=G(u), u\in [-c,c]$.
\end{remark}

Let us introduce the main results of this paper. To fix ideas, assume that the initial datum $0 \leq u_0 \in  L^\infty(\R)$ has compact support and mass equal to $M$. We also assume that $u_0$ has jump discontinuities only at the boundaries of its support, and the part of the interior of the support in which $u_0$ vanishes is a set of null measure. Finally, we assume that $\Phi$ satisfies \eqref{H1}--\eqref{H2}. Under these hypotheses, the existence of a unique entropic solution (see Theorem \ref{EUTEparabolic}) of system \eqref{s} is guaranteed by some previous analysis, for instance that carried out in \cite{ARMAsupp,ACM2005,CCCSS-SUR}. However, all the results of this paper are valid in the general framework of distributional solutions. Our results are focused on the study of three qualitative properties of these solutions:
\begin{enumerate}
\item {\em Regularity properties:} There is a finite time from which the solutions are regular inside their support. The idea is to study an associated dual problem for which, with the help of sub and super solutions as well as uniformly elliptic estimates, we can check the regularity properties of the solutions. These arguments will be developed in Section 2 and have been inspired in \cite{CCCSS-SUR,CCCSS-EMS,Carrillo}.
\item {\em Possible Blow--up and dynamic of the support:} We analyze the dynamic properties of the evolution of the support of the solution  in terms of the three parameters of the system: the mass $M$, the flux--saturated characteristic speed $c$, and the chemoattractant sensitivity constant $a$. We characterize the length of the support, $\ell(t)$, by means of the following function of time: $\ell(t)= \ell(0) +(2c-aM)t$.  In particular, if $M>\frac{2c}a$, there exists a possible blow--up time 
$ 
\displaystyle T^* = \frac{\ell(0)}{aM -2c}
$
at which the solution could be concentrated on a Dirac mass with mass $M$; this time $T^*$ is sharp in the case where the solution is not continuous, i.e. the ends of the solution support do not touch to zero. In the opposite case $M<\frac{2c}a$, the support of the solution grows indefinitely over time. Also, we prove in Section 3 that the velocity of distribution of the mass around the center of the support is given in terms of the mean first moment of the solution. These results are summarized in Section 3.
\item {\em Traveling waves:} We show that there are bounded traveling wave profiles both continuous and singular with jump discontinuities at the boundaries of the support of the solution, of the type ${u}(t,x)=u(x-\sigma t)$, where $u$ has support in $[\xi_-, \xi_+]$. Then, we prove that the traveling wave speed is described in terms of the flux--saturated velocity $c$, the  chemoattractant sensitivity constant $a$ and the mean first moment of $u$
\[
\displaystyle \bar{u} = \frac 1 {\ell(0)}  \int_{\xi_-}^{\xi_+} s\, u(s)\, ds,
\]
 as follows $
 \bar{\mu}= \frac{\xi_+M}{\ell(0)}  -\bar{u}.
 $
Since the support of the solutions are assumed compact and the mass is finite, the shape of the patterns here are more of  solitonic type. The analysis of these results are the aim of Section 4.
\end{enumerate}

Related to these results, although with different techniques, the following Keller--Segel initial--boundary value problem was analyzed in \cite{BW2,BW1}:
\begin{eqnarray*}
\displaystyle \frac{\partial u}{\partial t}&=&  \frac{\partial}{\partial x}\left( \frac{|u|  \frac{\partial u}{\partial x}}{\sqrt{u^2+ \left|\frac{\partial u}{\partial x}\right|^2}} \right) - a \frac{\partial}{\partial x}\left(
\frac{|u|  \frac{\partial S}{\partial x}}{\sqrt{1+ \left|\frac{\partial S}{\partial x}\right|^2}} \right),\\
\displaystyle {\frac{\partial^2 S}{\partial x^2}} &=  & u - \mu, \qquad \mu=\frac{1}{|\Omega |}\int_\Omega u_0(x) dx,
\end{eqnarray*}
with no--flux boundary conditions in balls $\Omega$. It was proved that the solutions are global and bounded if $m< m_c $, where $m_c= \left( a^2 -1\right)^{-1/2}$ if $a>1$, and $m_c = + \infty$ if $a \leq1$.

As a final remark, we would like to point out that the results can be extended to a broader context, particularly if the influence of the chemoattractant is given through a mean curvature type operator as in \eqref{optimal} or by the extended equation for the chemoattractant as established in the following remark.
\begin{remark}\label{r2}
In \eqref{s} we have considered a simplified  equation for the chemoattractant, but our analysis can be  extended to a more general class of equations of the type
$$
\displaystyle \delta{\frac{\partial S}{\partial t}}= {\frac{\partial^2 S}{\partial x^2}}+b u-\gamma S,\   x\in \R, \; t>0,
$$
being $ b, \delta$ and $\gamma$ positive constants.
For instance, the change of variable $W=e^{-(\gamma /\delta)t}S$ transforms the system (\ref{s})  into 
$$
\left\{\begin{array}{ll}
\begin{array}{l}
\displaystyle {\frac{\partial u}{\partial t}}= \frac{\partial}{\partial x} \left(u\Phi\left(u^{m-1}{\frac{\partial u}{\partial x}}\right) -ae^{(\gamma /\delta)t}{\frac{\partial W}{\partial x}}u \right),\\ \\
\displaystyle \delta{\frac{\partial W}{\partial t}}= {\frac{\partial^2 W}{\partial x^2}}+be^{-(\gamma /\delta)t} u,\end{array}&\quad  x\in \R, \; t>0.
\end{array}\right.
$$
The choice $b=1$ and the formal limits  $\delta \ll1$ and $\gamma /\delta \ll1$ allow to consider \eqref{s}. 
\end{remark}

\section{The dual problem}

Let us assume that  we are in the context of existence and uniqueness of entropic solutions given by Theorem \ref{EUTEparabolic} in Section 3, i.e. for any initial datum $0 \leq u_0 \in L^1(\R^d)\cap L^\infty(\R^d)$
there exists $T^d
> 0$ and a unique entropy solution $u$ of
(\ref{s}) in $Q_{T^d} = (0,T^d)\times \R^d$  such that $u(0) = u_0$. The index $^d$ refers to the dual problem.

The idea of this section is to analyze the regularity of solutions via an auxiliary dual problem by using a transformation called ``the mass coordinate'' of Lagrange \cite{MePuSh97}. This dual problem has some regularity properties that are typical of uniformly ellip\-tic operators of second order.  Lagrange transformations of this type are of relevance  in dealing with free boundary problems for nonlinear parabolic PDEs because the support transforms into a known domain, see \cite{Gurtin1984,MePuSh97} for references. 
This change of variables was  applied previously in \cite{CCCSS-SUR,CCCSS-EMS,Carrillo}.

Let $(u,S)$ be a solution of the system  (\ref{s}) 
such that  $u(t,x)\geq 0$. Define   $M:=\int_\R u(0,x)\, dx$, which is a constant of motion, and consider $\varphi: [0,+\infty)\times (0,M)\to \R$ given by
\begin{equation}\label{m}
\int_{-\infty}^{\varphi (t,\eta)} u(t,x)\, dx =\eta, \qquad \eta\in (0,M).
\end{equation}

Taking derivatives with respect to  $t$ in (\ref{m}),  and using the first equation in (\ref{s}), we deduce
\begin{eqnarray*}
0 &= & u(t, \varphi(t, \eta)){\frac{\partial \varphi}{\partial t}}(t, \eta) + \int_{-\infty}^{\varphi (t,\eta)} {\frac{\partial u}{\partial t}}(t,x)\, dx 
\\
&= &u(t, \varphi(t, \eta)){\frac{\partial \varphi}{\partial t}}(t, \eta) + \int_{-\infty}^{\varphi (t,\eta)} \frac{\partial}{\partial x}\left[u\Phi(u^{m-1}{\frac{\partial u}{\partial x}}) -a\left({\frac{\partial S}{\partial x}}u\right)\right] \, dx
\\
&=& u(t, \varphi(t, \eta)) \left[{\frac{\partial \varphi}{\partial t}}(t, \eta) + \Phi(u^{m-1}(t, \varphi(t,\eta)){\frac{\partial u}{\partial x}}(t, \varphi(t,\eta))-a{\frac{\partial S}{\partial x}}(t, \varphi(t,\eta))\right].
\end{eqnarray*}

Therefore, as long as $u(t, \varphi(t,\eta))>0$ is satisfied, we have
\begin{equation}\label{1phi}
{\frac{\partial \varphi}{\partial t}}((t,\eta))+  \Phi(u^{m-1}(t, \varphi(t,\eta)){\frac{\partial u}{\partial x}}(t, \varphi(t,\eta))=a\frac{\partial S}{\partial x}(t, \varphi(t,\eta)).
\end{equation}

Now,  taking  derivatives  in (\ref{m}) with respect to $\eta$, we find
$$
u(t, \varphi(t, \eta))\frac{\partial \varphi}{\partial \eta}(t, \eta) =1.
$$

Setting $\frac{\partial \varphi}{\partial \eta}(t, \eta)= v(t,\eta)$ and using the previous equality, we obtain $v(t,\eta)=(u(t, \varphi(t,\eta)))^{-1}$,  as long as $u(t, \varphi(t,\eta))>0$. Then, in this context we have
$$
\frac{\partial v}{\partial \eta}(t,\eta)= - {\frac{\partial u}{\partial x}}(t, \varphi(t,\eta))\frac{\partial \varphi}{\partial \eta}(t,\eta)(u(t, \varphi(t,\eta)))^{-2} = -{\frac{\partial u}{\partial x}}(t, \varphi(t,\eta))(v(t, \eta))^3,
$$
namely, ${\frac{\partial u}{\partial x}}(t, \varphi(t,\eta))=-\frac{\partial v}{\partial \eta}(t,\eta)(v(t, \eta))^{-3}.$ 

Substituting in (\ref{1phi}) and differentiating with respect to $\eta$, we find
$$
\frac{\partial v}{\partial t}(t,\eta)+\frac{\partial}{\partial \eta}\left[\Phi(-\frac{\partial v}{\partial \eta}(t,\eta)(v(t, \eta))^{-2-m})\right] = a\frac{\partial^2 S}{\partial x^2}(t, \varphi(t,\eta))v(t, \eta).
$$
Taking into account the second equation in (\ref{s}) and the even character of  $\Phi$, it can be checked that $v(t,\eta)$ satisfies
\begin{equation}\label{d}
\frac{\partial v}{\partial t}= \frac{\partial}{\partial \eta}\left[\Phi\left(\frac{\frac{\partial v}{\partial \eta}}{v^{2+m}}\right)\right] - a, \qquad t>0, \; \eta\in (0,M).
\end{equation}

We are looking for solutions of (\ref{d}) such that 
\begin{equation}\label{cc}
\displaystyle  \Phi\left(\frac{\frac{\partial v}{\partial \eta}(t,0)}{v^{2+m}(t,0)}\right)=-c, \quad \Phi\left(\frac{\frac{\partial v}{\partial \eta}(t,M)}{v^{2+m}(t,M)}\right)=c, \qquad t>0. \end{equation}

\begin{theorem}\label{tb1}
Let $v_0\in W^{1,\infty}(0,M)$, such that $v_0\geq \sigma_1>0$. Then, there exist  $T^d$ with $0<T^d<+\infty$, and a regular solution $v$ of (\ref{d}) with initial datum  $v(0,\eta)=v_0(\eta)$, which is defined in  $(0,T^d)\times (0,M)$  and verifies the boundary conditions (\ref{cc}).
\end{theorem}
\begin{proof} Let $\varepsilon>0$ and $T>0$ be fixed. We first consider the following approximate problem:

\begin{equation}\label{da}
\frac{\partial v}{\partial t}=\frac{\partial}{\partial \eta} \left[\Phi\left(\frac{\frac{\partial v}{\partial \eta}}{v^{2+m}}\right)\right] -a + \varepsilon \frac{\partial^2 v}{\partial \eta^2}, \qquad t \in (0,T), \; \eta\in (0,M),
\end{equation}
\begin{equation}\label{cca}
\Phi\left(\frac{\frac{\partial v}{\partial \eta}(t,\eta)}{v^{2+m}(t,\eta)}\right)+ \varepsilon \frac{\partial v}{\partial \eta}(t,\eta)=(-1)^{1-\frac \eta M}(c- \varepsilon^\kappa), \; \eta\in\{0, M\}, \; t\in (0,T).
\end{equation}
with $\kappa >0$ to be determined.

Then, we proceed in several steps:

\medskip

\noindent
\underline{Step  1}: {\em $L^{\infty}$ bounds independent of $\varepsilon$}.
For every $\varepsilon >0$ we define
$$
h( \eta):= \frac{1}{A}\left(G(A\eta-c+\varepsilon^\lambda ) - G(c-\varepsilon^\lambda)\right),
$$
with  $A=\frac{2(c-\varepsilon^\lambda)}{M},$
and $\lambda >0$ to be chosen.
Hypotheses (\ref{H1}) y (\ref{H2}) imply that $h \in C^3[0,M] $ is symmetric with respect to $\eta=M/2, \; h(0)=0$, and 
$$
h(\eta)\geq h(M/2)=\frac{-1}{A}G(c-\varepsilon^\lambda)\to \frac{-MG(c)}{2c}, \; \mbox{ as } \varepsilon \to 0.
$$
 Therefore, $h(\eta)>  \frac{-MG(c)}{c}, \; \eta \in [0,M]$, as $\varepsilon$ is small enough.

We are going to build a super-solution of (\ref{da})-(\ref{cca}) in the form
$$
\overline{V}(t, \eta)=B(t) +h(\eta), $$
being $B(t)$ a strictly increasing function to be determined, with 
$$B(0) \> \max \left\{ \frac{MG(c)}{c}, \|v_0\|_{\infty}\right\}.$$

As a consequence of Remark \ref{R1}, we have 
\begin{equation}\label{h'}
\left|\frac{\partial \overline{V}}{\partial \eta} (t,\eta)\right|= | h'(\eta)| =|g(A\eta -c+\varepsilon^\lambda)| \leq g (c-\varepsilon^\lambda)\leq C_1  \varepsilon^{(-\beta +1) \lambda},
\end{equation}
and
\begin{equation}\label{h''}
\left|\frac{\partial^2 \overline{V}}{\partial \eta^2} (t,\eta)\right|= |h''(\eta)|= A g'(A\eta -c +\varepsilon^\lambda)\leq C_2 \varepsilon^{-\beta \lambda},
\end{equation}
as $\varepsilon$ is small enough, for all $t>0$ and $\eta \in (0,M)$, where for convenience we have put $\beta=\frac{\alpha+1}{\alpha}$ from now on. This yields to
$$
\left|\varepsilon \frac{\partial^2 \overline{V}}{\partial \eta^2} (t,\eta)\right|\to 0, \quad \mbox{ as } \quad \varepsilon \to 0 ,
$$
provided that $\lambda< \frac 1 \beta$. In particular, it follows $\left|\varepsilon \frac{\partial^2 \overline{V}}{\partial \eta^2} (t,\eta)\right|\leq C_3, \; t>0, \;\eta \in (0,M)$.

On the other hand, we have
$$
\frac{\partial}{ \partial \eta}\left(\Phi\left( \frac{\frac{\partial \overline{V}}{\partial\eta}}{\overline{V}^{2+m}}\right)\right) = \Phi' \left( \frac{h'}{\overline{V}^{2+m}}\right)\left[ \frac{h''}{\overline{V}^{2+m}}-\frac{(2+m)(h')^2}{\overline{V}^{3+m}}\right].
$$
Since  $\Phi'$ and $\overline{V}$ are both positive functions, we deduce
\begin{equation}\label{phi}
\frac{\partial}{ \partial \eta}\left(\Phi\left( \frac{\frac{\partial \overline{V}}{\partial\eta}}{\overline{V}^{2+m}}\right)\right) \leq \Phi' \left( \frac{h'}{\overline{V}^{2+m}}\right)\frac{h''}{\overline{V}^{2+m}} \leq C_4,
\end{equation}
where we have used $h''(\eta)= Ag'(A\eta-c+\varepsilon^\lambda)=\frac{A}{\Phi'(g(A\eta-c+\varepsilon^\lambda)}$. This implies
$$
\Phi' \left( \frac{h'}{\overline{V}^{2+m}}\right)\frac{h''}{\overline{V}^{2+m}}= \frac{A}{\overline{V}^{2+m}}\frac{\Phi' \left( g(A\eta-c+\varepsilon^\lambda)/\overline{V}^{2+m}\right)}{\Phi'(g(A\eta-c+\varepsilon^\lambda)},
$$
and $\overline{V}>\sigma_1$.

Then, as a consequence of   (\ref{h''}) and (\ref{phi}), and choosing  $\lambda<\frac 1 \beta$, we find 
$$
\frac{\partial}{ \partial \eta}\left(\Phi\left( \frac{\frac{\partial \overline{V}}{\partial\eta}}{\overline{V}^{2+m}}\right)\right) -a +\varepsilon \frac{\partial^2\overline{V}}{\partial\eta^2} \leq C_5,  \qquad t>0, \;\eta \in (0,M),
$$
as $\varepsilon$ is small enough. Now, we choose  $B(t)=C_5 t + B(0)$ in such a way that $\overline{V}(t,\eta)$ verifies
$$
\frac{\partial \overline{V}}{\partial t}\geq \frac{\partial }{\partial \eta} \left(\Phi\left( \frac{\frac{\partial \overline{V}}{\partial\eta}}{\overline{V}^{2+m}}\right)\right) -a +\varepsilon \frac{\partial^2 \overline{V}}{\partial \eta^2} ,  \qquad t>0, \;\eta \in (0,M).
$$
Furthermore, we have
$$
\Phi\left( \frac{\frac{\partial \overline{V}}{\partial\eta}}{\overline{V}^{2+m}}\right) + \varepsilon \left.  \frac{\partial \overline{V}}{\partial \eta} \right|_{\eta=0}= \Phi\left(\frac{-g(c-\varepsilon^\lambda)}{B(t)^{2+m}}\right) - \varepsilon g(c-\varepsilon^\lambda).
$$
As a consequence of (\ref{H1}) and (\ref{H2}), there exists $K>0$ such that 
\begin{equation}\label{c}
\Phi(u) \geq c-K/u, \; u>0.
\end{equation}
Therefore,  fixing $T>0$ the following estimate
\begin{eqnarray*}
\displaystyle \Phi\left( \frac{\frac{\partial \overline{V}}{\partial \eta}}{\overline{V}^{2+m}}\right) + \varepsilon \left.  \frac{\partial \overline{V}}{\partial \eta} \right|_{\eta=0} &\leq& -c +K \frac{B(t)^{2+m}}{g(c-\varepsilon^\lambda)}- \varepsilon g(c-\varepsilon^\lambda) \\ &\leq& -c +C_6\varepsilon^{\lambda(\beta -1)} - C_7\varepsilon^{1-\lambda(\beta -1)},
\end{eqnarray*}
holds, for all $t\in [0,T]$.
It is enough to choose $\lambda < \frac 2 3(\beta -1)$ and set $\kappa=\frac \lambda 2 (\beta -1)$ in order to obtain 
\begin{eqnarray*}
\displaystyle \Phi\left( \frac{\frac{\partial \overline{V}}{\partial \eta}}{\overline{V}^{2+m}}\right) + \varepsilon \left.  \frac{\partial \overline{V}}{\partial \eta} \right|_{\eta=0}\leq -c+\varepsilon^\kappa, \; t\in [0,T].
\end{eqnarray*}
Analogously, we also have
\begin{eqnarray*}
\Phi\left( \frac{\frac{\partial \overline{V}}{\partial \eta}}{\overline{V}^{2+m}}\right) + \varepsilon \frac{\partial \overline{V}}{\partial \eta} \mid_{\eta=M}\geq c-\varepsilon^\kappa, \quad t\in [0,T].
\end{eqnarray*}

In conclusion,  the function $\overline{V}(t,\eta)$ is  a super-solution of (\ref{da})--(\ref{cca}) in $[0,T]\times (0,M), \; T>0$, as long as  $\lambda<\min\{\frac{\alpha + 1} \alpha, \frac 2 {3\alpha}\}$.

It is clear that the function $\underline{V}(t)=\sigma_1-at$ is a sub-solution, which is positive in $[0,T^d)$, with $T^d=\sigma_1/a$.

In particular, if $v$ is a solution of (\ref{da})--(\ref{cca}) with $v(0,\eta)=v_0(\eta)$, then, fixing $0<T<T^d$, we have
$$
0<\sigma_1-aT\leq \underline{V}(t)\leq v(t,\eta)\leq \overline{V}(t,\eta), \quad t\in [0,T],  \; \eta\in (0,M).
$$

\medskip

\noindent
\underline{Step  2}: {\em $L^{1}$  bounds of $\frac{\partial v^p}{\partial \eta}$ independent of $\varepsilon$}.
Integrating (\ref{da}) and using the boundary conditions (\ref{cca}), we obtain
\begin{equation}\label{f1}
\int_0^M v(t,\eta)\, d\eta = \int_0^M v(0,\eta)\, d\eta -aMt+2(c-\varepsilon^\kappa)t.
\end{equation}
Therefore, taking into account the estimates obtained in the previous step, we deduce that $v(t,\cdot)\in L^p(0,M)$, for any $1\leq p \leq \infty$ and $t\geq 0$.
On the other hand, fixing $p\geq 1$ and $t>0$, multiplying (\ref{da}) by $v^p$ and integrating in $(0,t) \times (0,M)$, it follows
\begin{eqnarray*}
 \int_0^t\int_0^M \Phi\left( \frac{\frac{\partial v}{\partial \eta}}{v^{2+m}}\right) \frac{\partial v^p}{\partial \eta} \, d\eta \, dt &+& \varepsilon p \int_0^t\int_0^M v^{p-1} \left(\frac{\partial v}{\partial \eta}\right)^2 \, d\eta \, dt 
 \\ &=&
\frac{1}{p+1}\int_0^M \left( v^{p+1}(0,\eta) - v^{p+1}(t,\eta)\right)\, d\eta 
\\
& & + \ (c-\varepsilon^\kappa)\int_0^t \left( v^{p}(s,0) + v^{p}(s,M)\right) \, ds -aMt. \label{ep0sop}
\end{eqnarray*}
Using (\ref{c}), the following lower bound 
$$ \int_0^t\int_0^M \Phi\left( \frac{\frac{\partial v}{\partial \eta}}{v^{2+m}}\right)\frac{\partial v^p}{\partial \eta} \, d\eta \, dt \geq c  \int_0^t\int_0^M \left|\frac{\partial v^p}{\partial \eta}\right| \, d\eta \, dt - Kp \int_0^t\int_0^M v^{p+m+1} \, d\eta \, dt,
$$
holds, and as a consequence we have
\begin{equation}
 \int_0^t\int_0^M \left|\frac{\partial v^p}{\partial \eta}\right| \, d\eta \, dt \leq C(t, p), \quad \mbox{for all} \; p\in [1,\infty), \; \mbox{and}\; t \geq 0.
 \label{distribucional}
\end{equation}
\medskip

\noindent 
\underline{Step  3}: {\em Existence of regular solutions to the Cauchy problem}. As a consequence of the a priori estimates obtained in Step 1,  we can use the regular flux 
$$
a_T(v,\frac{\partial v}{\partial \eta}):=\Phi\left( \frac{\frac{\partial v}{\partial \eta}}{\sup\{\sigma_1-aT,v^{2+m}\}}\right),
$$
for a fixed  $0<T<T^d$, in order to prove the existence of regular solutions of (\ref{da})--(\ref{cca}) in $[0,T]\times (0,M)$ following  classical results, for instance Theorem 13.24 in \cite{Li}. We remark that the initial condition $v_0$ has to be modified if it does not check the compatibility conditions (\ref{cca}). To make this modification we will use the following result.
\begin{lemma}\label{v0} Let $\delta_0\in(0,M)$ fixed. Then, there exists  $\tilde{\varepsilon}\in (0,\delta_0)$ such that  the following estimates 
\begin{eqnarray*}
\Phi\left(\frac{B(\varepsilon)}{[v_0 (0)]^{2+m}}\right) + \varepsilon B(\varepsilon) &>&c-\varepsilon^\kappa, \\
\Phi\left(\frac{B(\varepsilon)}{[v_0 (\delta_0)+\delta_0B(\varepsilon)]^{2+m}}\right) + \varepsilon B(\varepsilon) &<&c-\varepsilon^\kappa,
\end{eqnarray*}
hold for  $0<\varepsilon<\tilde{\varepsilon},$ and $ B(\varepsilon):=\frac{c-\varepsilon^\kappa}{2\varepsilon}$.
\end{lemma}
\proof It is enough to observe that $B(\varepsilon) \to +\infty$ as $\varepsilon \to 0$.

As a consequence, for each $\varepsilon<\tilde{\varepsilon}$, there exists $\delta(\varepsilon) \leq \delta_0$ such that
\begin{equation}\label{ev0}
\Phi\left(\frac{B(\varepsilon)}{[v_0 (\delta)+\delta B(\varepsilon)]^{2+m}}\right) + \varepsilon B(\varepsilon) = c-\varepsilon^\kappa.
\end{equation}
The function
$$
v_{0,\varepsilon}(\eta):=\left\{\begin{array}{ll}
v_0(\delta)+B(\varepsilon)(\delta-\eta), & \eta\in[0,\delta],\\
v_0(\eta), & \eta\in(\delta, M],
\end{array}\right.
$$
fulfills the boundary condition (\ref{cca}) in $\eta=0$. A similar construction allows modifying the initial condition $v_0$ so that $v_{0,\varepsilon}$ verifies (\ref{cca}) in $\partial(0,M)$ and
$$
\sup_{0<\varepsilon<\tilde{\varepsilon}} \varepsilon \|v'_{0,\varepsilon}\|_\infty <\infty.
$$

Denoting $v_\varepsilon$ the solution of (\ref{da})-(\ref{cca}) that meets the initial condition $v_\varepsilon(0,\eta)=v_{0,\varepsilon}(\eta)$, we deduce that $v_\varepsilon$ has first H\"older--continuous derivatives  up to the boundary. In addition, setting $g=v_{\varepsilon\eta\eta},\, v_{\varepsilon t}$, we have
$$
\sup_{\eta_1\neq \eta_2}\left\{ \min (d((t,\eta_1), \mathcal{P}), d((s,\eta_2), \mathcal{P}))^{1-\nu}\frac{|g(t,\eta_1)-g(s,\eta_2)|}{(|t-s|+|\eta_1-\eta_2|^2)^{\frac \gamma 2}}\right\}<\infty,
$$
for some $\nu, \gamma >0$, where $\mathcal{P}$ is the  parabolic boundary of $(0,T)\times (0,M)$ and $d(\cdot, \mathcal{P})$ is the distance to $\mathcal{P}$. On the other hand, by classical interior regularity results (see \cite{Ladyparabolico}, Chapter V, Theorem 3.1), the solution is infinitely smooth in the interior of the domain. Here, the smoothness bounds depend on $\varepsilon$.

\medskip

\noindent
\underline{Step 4}: {\em An $\varepsilon$--uniform Lipschitz estimate   for $v_\varepsilon$}. 
For simplicity, let us write $v$ instead of $v_\varepsilon$. Given a regular, non-negative test function $\varphi$ with compact support contained in $(0, M)$, we define  
$$\omega= \frac 1 2 \left(\frac{\partial v}{\partial\eta}\right)^2\varphi^2.$$ 
Taking into account that  $v$ solves (\ref{da}), it follows
\begin{eqnarray*}
\frac{\partial \omega}{\partial t}&=&
 \frac{\partial v}{\partial \eta}\varphi^2 \left\{\left(\Phi'\left(\frac{\frac{\partial v}{\partial\eta}}{v^{2+m}}\right)\left[\frac{\frac{\partial^2v}{\partial\eta^2}}{v^{2+m}}-\frac{(2+m)\left(\frac{\partial v}{\partial\eta}\right)^2}{v^{3+m}}\right]\right)_\eta + \varepsilon \frac{\partial^3v}{\partial\eta^3}\right\}
\\
&=& \frac{\partial v}{\partial \eta}\varphi^2 \Bigg\{\Phi''\left(\frac{\frac{\partial v}{\partial\eta}}{v^{2+m}}\right)\left[\frac{\frac{\partial^2v}{\partial\eta^2}}{v^{2+m}}-\frac{(2+m)\left(\frac{\partial v}{\partial\eta}\right)^2}{v^{3+m}}\right]^2\Bigg\}  
\\
&& + \, \frac{\partial v}{\partial \eta}\varphi^2 \Bigg\{\Phi'\left(\frac{\frac{\partial v}{\partial\eta}}{v^{2+m}}\right)\left[\frac{\frac{\partial^3v}{\partial\eta^3}}{v^{2+m}}-\frac{3(2+m)\frac{\partial^2v}{\partial\eta^2}\frac{\partial v}{\partial\eta}}{v^{3+m}}+ \frac{(2+m)(3+m)\left(\frac{\partial v}{\partial\eta}\right)^3}{v^{4+m}}\right] 
\\
&&+ \, \varepsilon \frac{\partial^3v}{\partial\eta^3}\Bigg\}.
\end{eqnarray*}
Since $s\Phi''(s)<0$, the first term of the right--hand side is negative and, therefore, we deduce
\begin{eqnarray*}
\displaystyle \frac{\partial \omega}{\partial t} &\leq& \frac{\partial v}{\partial \eta}\varphi^2 \Phi'\left(\frac{\frac{\partial v}{\partial\eta}}{v^{2+m}}\right)\left[\frac{\frac{\partial^3v}{\partial\eta^3}}{v^{2+m}}-\frac{3(2+m)\frac{\partial^2v}{\partial\eta^2}\frac{\partial v}{\partial\eta}}{v^{3+m}}+ \frac{(2+m)(3+m)\left(\frac{\partial v}{\partial\eta}\right)^3}{v^{4+m}}\right] \\
\displaystyle && \, + \frac{\partial v}{\partial \eta}\varphi^2\varepsilon \frac{\partial^3v}{\partial\eta^3}.
\end{eqnarray*}
On the other hand, we have
\begin{eqnarray*}
\displaystyle \frac{\partial \omega}{\partial\eta} &=& \frac{\partial v}{\partial\eta} \frac{\partial^2v}{\partial\eta^2}\varphi^2 + \left(\frac{\partial v}{\partial\eta}\right)^2\varphi\varphi', \\
\displaystyle \frac{\partial^2 \omega}{\partial\eta^2}&=& \left(\frac{\partial^2v}{\partial\eta^2}\right)^2\varphi^2 + \frac{\partial v}{\partial\eta} \frac{\partial^3v}{\partial\eta^3}\varphi^2+ 4 \frac{\partial v}{\partial\eta} \frac{\partial^2v}{\partial\eta^2}\varphi\varphi' +  \left(\frac{\partial v}{\partial\eta}\right)^2\varphi'^2 +  \left(\frac{\partial v}{\partial\eta}\right)^2\varphi\varphi''.
\end{eqnarray*}
In particular, we find
\begin{eqnarray*}
\displaystyle \frac{\partial v}{\partial\eta} \frac{\partial^2v}{\partial\eta^2}\varphi^2 &=& \frac{\partial \omega}{\partial\eta} - \left(\frac{\partial v}{\partial\eta}\right)^2\varphi\varphi_{\eta}\leq \frac{\partial \omega}{\partial\eta} + \omega +\frac{ 1}{ 2 } \left(\frac{\partial v}{\partial\eta}\right)^2\varphi'^2,
\\
\displaystyle \frac{\partial v}{\partial\eta} \frac{\partial^3v}{\partial\eta^3}\varphi^2 &=& \frac{\partial^2 \omega}{\partial\eta^2}-\left(\frac{\partial^2v}{\partial\eta^2}\right)^2\varphi^2 - 4 \frac{\partial v}{\partial\eta} \frac{\partial^2v}{\partial\eta^2}\varphi\varphi' -  \left(\frac{\partial v}{\partial\eta}\right)^2\varphi'^2 +  \left(\frac{\partial v}{\partial\eta}\right)^2\varphi\varphi''
\\
\displaystyle &\leq& \frac{\partial^2 \omega}{\partial\eta^2}- \left(\frac{\partial^2v}{\partial\eta^2}\right)^2\varphi^2 +4\left(\frac{\partial v}{\partial\eta}\right)^2\varphi'^2 +\frac{\partial^2v}{\partial\eta^2}^2\varphi^2 
\\
\displaystyle && -  \left(\frac{\partial v}{\partial\eta}\right)^2\varphi'^2 +  \left(\frac{\partial v}{\partial\eta}\right)^2\varphi\varphi'' 
\\
\displaystyle &\leq&\frac{\partial^2 \omega}{\partial\eta^2}+\omega + \left(\frac{\partial v}{\partial\eta}\right)^2 \left(3\varphi'^2+ \varphi''^2\right).
\end{eqnarray*}
Then, combining the above inequalities and taking into account that $\Phi'(s)$ and $v$ are positive functions, we deduce
\begin{equation}\label{omega2}
\displaystyle \frac{\partial \omega}{\partial t} \leq A(t,\eta) \frac{\partial^2 \omega}{\partial\eta^2} + B(t,\eta)\frac{\partial \omega}{\partial\eta}+C(t,\eta)\omega +f(t, \eta),
\end{equation}
with
$$
\begin{array}{l}
\displaystyle A(t,\eta)= \frac{1}{v^{2+m}} \Phi'\left(\frac{\frac{\partial v}{\partial\eta}}{v^{2+m}}\right) +\varepsilon,\\
\displaystyle B(t,\eta)= \frac{3(2+m)|\frac{\partial v}{\partial\eta}|}{v^{3+m}}\Phi'\left(\frac{\frac{\partial v}{\partial\eta}}{v^{2+m}}\right),\\
\displaystyle C(t,\eta)=  A(t, \eta) + B(t,\eta)+ \frac{2(2+m)(3+m)\left(\frac{\partial v}{\partial\eta}\right)^2}{v^{4+m}}\Phi'\left(\frac{\frac{\partial v}{\partial\eta}}{v^{2+m}}\right),\\
\displaystyle f(t,\eta)=\left(\frac{\partial v}{\partial\eta}\right)^2\left(A(t,\eta)\left(3 \varphi'^2+ \frac 1 2\varphi''^2\right)+\frac 1 2 B(t,\eta)\varphi'^2\right).
\end{array}
$$
As a consequence of Step 1 and the hypotheses on $\Phi$, which in particular imply that $\Phi '(s) s^2$ is bounded, it is easily deduced that the functions $A, B $ and $ C $ are bounded independently of $\varepsilon$, and $ A $, and therefore also  $C$, are strictly positive. 

The term A$\left(\frac{\partial v}{\partial\eta}\right)^3\Phi'\left(\frac{\frac{\partial v}{\partial\eta}}{v^{2+m}}\right)$ that appears in the definition of $f$ is bounded because $\beta\leq 3/2$, since $\alpha \geq 2$ due to (\ref{H2}), and $\Phi'(s)s^3$ is bounded. We need to find an estimate of $\varepsilon \left(\frac{\partial v}{\partial\eta}\right)^2$ independent of $\varepsilon$. To this aim, if we put $\varphi=1$ in (\ref{omega2}), we get inside of $(0,M)$
$$
\displaystyle \frac{\partial \omega}{\partial t} \leq A(t,\eta) \frac{\partial^2 \omega}{\partial\eta^2} + B(t,\eta)\frac{\partial \omega}{\partial\eta}+C(t,\eta)\omega.
$$

 On the other hand, taking into account the boundary conditions (\ref{cca}), there follows 
$$
c\left|\frac{\partial v}{\partial\eta} \right| - v^{2+m} +\varepsilon \left(\frac{\partial v}{\partial\eta}\right)^2 \leq  \Phi'\left(\frac{\frac{\partial v}{\partial\eta}}{v^{2+m}}\right) \left| \frac{\partial v}{\partial\eta}\right| +\varepsilon \left(\frac{\partial v}{\partial\eta}\right)^2 \leq (c-\varepsilon^\kappa)\left|\frac{\partial v}{\partial\eta}\right|.
$$

Hence, $\varepsilon \left(\frac{\partial v}{\partial\eta}\right)^2(t,\cdot)\leq v^{2+m}(t,\cdot)$ and consequently
$$
\varepsilon \omega(t, \cdot)\leq C, \;\mbox{on}\; (0,T^d)\times \partial (0,M).
$$
Then, the maximum principle shows that
\begin{equation}\label{mp}
\varepsilon\left\|\frac{\partial v}{\partial\eta}(t,\cdot)\right\|^2_{L^\infty(0,M)}\leq \tilde{C}, \;\mbox{in}\; (0,T^d)\times  (0,M),
\end{equation}
where $\tilde{C}$ is a constant independent of $\varepsilon$.
Therefore, we can deduce uniform bounds for $f(t,\eta)$ independent of $\varepsilon$ in $(0,T) \times (0,M)$, for any $0<T<T^d$. 

We have obtained local Lipschitz bounds on $\frac{\partial v}{\partial\eta}$ which are uniform in $\varepsilon$ and hold for $t \in (0,T^d)$.

\medskip

\noindent
\underline{Step 5}: {\em Conclusion}.  The smoothness results stated in Step 3 and the local uniform bounds on the gradient shown in Step 4 allow us to apply the classical interior regularity results in \cite{Ladyparabolico} in order to show uniform--in--$\varepsilon$ interior bounds for any space and time derivative of $v_\varepsilon$. A process of passage to the limit similar to that made  in \cite{CCCSS-SUR,CCCSS-EMS,Carrillo} concludes the proof.
\end{proof}

In order to simplify the notation from now on we consider $T=T^d$ in both dual and original problems.

\begin{remark}
 Let us comment some estimations obtained on the last proof, whose validity is independent of $\varepsilon$ and will be useful in the next section. 
 The limit as $\varepsilon\to 0$ in (\ref{f1}) leads to 
 \begin{equation}
\int_0^M v(t,\eta)\, d\eta = \int_0^M v_0(\eta)\, d\eta +(2c-aM)t.
\label{f1e0}
\end{equation}
 On the other hand choosing $p=1$ and $t=T^d$ in (\ref{distribucional}) we have
\begin{equation}
 \int_0^ {T^d}\int_0^M \left|\frac{\partial v}{\partial \eta}(t,\eta )\right| \,
 d\eta \, dt <\infty.
 \label{distribucionale0}
\end{equation}
 
\end{remark}

\section{Rankine--Hugoniot jump conditions and  possible blow--up}

Admissible bounded  solutions develop and transport jump discontinuities, but 
the structure of the singularities that a solution may eventually display is strongly restricted and is characterized by having a vertical profile that moves according to a  
Rankine--Hugoniot law.   
The Rankine--Hugoniot condition roughly says that singularity moves at speed 
\begin{equation}
\label{rhcond}
V(t)= \frac{\mathcal{F}(t,x^+)-\mathcal{F}(t,x^-)}{u(t,x^+)-u(t,x^-)},
\end{equation}
where ${\mathcal{F}}$ denotes the flux associated with the problem under consideration and $u$ is the density. 
Here, the  flux--saturated equations (including those of porous media type) are not much different from scalar conservation laws.

On the other hand,  the jump discontinuities fulfilling 
the Rankine--Hugoniot conditions can either spread out or move towards
a concentration of the mass depending on the constants of the system:
the total mass, the flux--saturated characteristic speed $c$, and the chemoattractant sensitivity constant $a$.
 If $u$ stops being bounded in $L^\infty$ we will say that a blow-up has occurred.

We are going to specify these concepts, in particular distributional and entropy solutions, in our specific case.

\subsection{Distributional framework and moving fronts}

Given an initial datum $u_0\in L^\infty(\R)$ with compact support, we intend to give a solution to the Cauchy problem meaning for the differential equation (\ref{s}) defined
in $\overline{Q_T}$, with $Q_T=]0,T[\times \R$, for some  $T>0$.  We  restrict ourselves to solutions with only two  propagation fronts,
that is, we are going to characterize two functions
$\sigma_+,\, \sigma_- \in C^1([0,T]), \; \sigma_- (t)<\sigma_+ (t), \,0\leq t\leq T$, 
and a solution $u : \overline{Q_T} \to \RR^+$
such that  $u(t,\cdot)\in L^{\infty}(\R)$, and $u(t,x)=0$ if $x\notin [\sigma_- (t), \sigma_+ (t)]$, a.e. $t\in ]0,T[$. Note that $ S: Q_T \to \R $
is being built from $ u $. That is, since  $u(t,\cdot )\in L^p(\R)$ for every $p\in ]1,+\infty[$ and a.e. $t\in ]0,T[$,
then it is possible to determine $S(t, \cdot) \in W^{2,p}(\R)$ by solving
\begin{equation}
-S_{xx}=u(t,x), \, x\in \R, \quad |S| \to 0, \mbox{ as } |x| \to \infty. 
\label{second}
\end{equation}
Once $S$ is known we give a distributional sense to (\ref{s}) by imposing 
\begin{equation}
 \int_0^T\int_{\R}u \frac{\partial \psi}{\partial t} \,dx\,dt -
\int_0^T\int_{\R}u\left[\Phi\left(u^{m-1}\frac{\partial u}{\partial x}\right) - a S_x\right] 
\frac{\partial \psi}{\partial x}\,dx\,dt  = 0,
\label{distr}
\end{equation}
for all $\psi \in C^1_{0}(Q_T)$. When $t=0$ belongs to the support of $\psi$, we also recover  the initial datum in the distributional formulation \eqref{distr}. 
The  integrals in \eqref{distr} make sense if
\begin{equation}
\int_0^T \int _{\R}\left | \frac{\partial u}{\partial x}(t,x) \right | dxdt<\infty.
\label{integral}
\end{equation} 
Also, we have    $u(t, \cdot)\in W^{1,1}(\sigma_- (t), \sigma_+ (t))$, a.e. $t\in ]0,T[$.

\begin{remark} Note that   $ \int _{\R} \left | \frac{\partial u}{\partial x}(t,x) \right | dx<\infty$, a.e. $t\in ]0,T[$, can be
seen as the measure of the variation
of the absolutely continuos part of $D_x u$ in the sense of distributions, for $u(t,\cdot)$. The singular
part can be written as $u(t,\sigma_-(t))^+  + u(t,\sigma_+(t))^-$, which makes sense since 
$W^{1,1}(\sigma_- (t), \sigma_+ (t))\subset C^0([\sigma_- (t), \sigma_+ (t)])$. Then, $u\in L^1(0,T; \, BV(\R))$. 
\end{remark}

As usual,  
$$
\Sigma=\{ (t,x)\in Q_{T} : x=\sigma_{\pm}(t) \} 
$$
denotes the singular set (points of discontinuity) of $u$. Let us prove the following result
\begin{theorem}
Let $(\xi_-, \xi_+)$ an open interval and  $u_0\in W^{1,\infty} (\xi_-, \xi_+)$  such that $\displaystyle \inf_{(\xi_-, \xi_+)}u_0(x)>0$. 
There exists  a time $T>0$ and a distributional solution $u$ of (\ref{s}) with two fronts defined in $\overline{Q_T}$,
which satisfies 
 $$
 u\in C^2 \left ( {Q}_{T}\setminus \Sigma \right )\cap  C^0  \left (\overline{ Q_{T}}\setminus \Sigma \right )
 $$
 and 
 $$
 u(0,x)=\left \{ \begin{array}{ll}
          u_0(x), & \mbox{ if }x\in (\xi_-, \xi_+),\\
          0, & \mbox{ otherwise. }
         \end{array}\right .
$$
 \end{theorem}

\begin{proof}  
Let  $v_0 : [0,M]\to (0,\infty)$ be defined by
 $v_0 = \frac{d}{d \eta}\varphi_0$, where  
$$
\displaystyle \int_{\xi_-}^{\varphi_0(\eta)} u_0(s)\, ds=\eta,
$$
then  $v_0\in W^{1,\infty}(0, M)$. Using  Theorem \ref{tb1} we can construct a solution $v(t,\eta)$ of
the  dual problem (\ref{d}), defined on $[0,T]\times [0,M]$, where $M=\displaystyle \int_{\xi_-}^{\xi_+}u_0(x)dx$. 

 Now we are going to define $\sigma_+,\sigma_- :[0,T]\to \R$ and $\varphi :[0,T]\times [0,M]\to \R$ such that 
 $\varphi (t,0)=\sigma_-(t)$ and $\varphi (t,M)=\sigma_+(t)$ in order to determine the solution $u$  solving 
$$
\displaystyle \int_{\sigma_-(t)}^{\varphi(t,\eta)} u(t,s)\, ds=\eta.
$$

Using (\ref{f1e0}) we obtain 
 $$
\displaystyle \sigma_+(t)-\sigma_-(t)=\int_0^M \frac{\partial \varphi}{\partial \eta}(t,s)\,
ds =\int_0^Mv(t, s)\, ds ={\int_0^Mv_0( s)\, ds +(2c-aM)t.}
$$
Note that $\displaystyle \int_0^Mv_0( s)\, ds= \displaystyle \int_0^M\frac{d}{d \eta}\varphi_0(s) ds=\xi_+-\xi_-$. Thus, we find a first relation between the evolution of both fronts
\begin{equation}
 \displaystyle \sigma_+(t)-\sigma_-(t)=\xi_+-\xi_- +(2c-aM)t.
 \label{sopmov}
\end{equation}
 We  remark that $\varphi_0$ is known, but  $\varphi(t,\eta)= \sigma_-(t)+\displaystyle \int _0^\eta v(t,s)ds,$ will not be known until we define $\sigma_-$ or 
 $\sigma_+$. 

A new function has to be defined in order to compute  $\sigma ^+, \sigma^-$, which can be obtained by integrating the second equation of (\ref{s}). First, we define
\begin{equation}
 \displaystyle \mu(t,x)=\int_{\sigma_-(t)}^x u(t,s)\, ds.
 \label{nuii}
\end{equation}
Thus, we have
\begin{equation}\label{r1}
\displaystyle -\frac{\partial S}{\partial x}(t,x)=\mu(t,x) - \bar{\mu}(t),
\end{equation}
where 
\begin{equation}
 \displaystyle \bar{\mu}(t)=\frac 1 {\sigma_+(t)-\sigma_-(t)}\int_{\sigma_-(t)}^{\sigma_+(t)} \mu(t,s)\, ds.
 \label{muval}
\end{equation}
Using the change of variable  $ s = \varphi (t, \eta) $ in (\ref{muval}) we infer
\begin{eqnarray} \label{muu}
 \bar{\mu}(t)=\frac{\displaystyle \int_0^M\mu (t,\varphi (t,\eta))v(t,\eta)d \eta}{\sigma_+(t)-\sigma_-(t)}=
\frac{\displaystyle \int_0^M \eta v(t,\eta)d \eta}{\xi_+-\xi_- +(2c-aM)t},
\end{eqnarray}
which allows us to define  $ \bar{\mu} $ in terms of the dual solution, and  thus deduce that $ \bar{\mu} $ is continuous. 

Finally, we compute $\sigma_+,\, \sigma_-$ by solving 
\begin{equation}
 \left \{\begin{array}{ll}
\sigma'_-(t)= -c+a\bar{\mu}(t), &\qquad  \sigma_+(0)=\xi_+,\\
\sigma'_+(t)= c-a(M-\bar{\mu}(t)),& \qquad \sigma_-(0)=\xi_-.
 \end{array}\right.
\label{desplaz}
\end{equation}
Thus, we obtain two $C^1$ functions.

Let us give a short explanation of how this expression was obtained. As usual, the propagation of singularities is always due to a saturation of the flux. This leads to 
\begin{equation}
 \frac{\partial u}{\partial x}(t, \sigma_-(t))=+\infty, \; \mbox{and }\; \frac{\partial u}{\partial x}(t, \sigma_+(t))=-\infty.
 \label{satura}
\end{equation}
Note that although $ u $ is bounded, this does not prevent its derivatives from being limited. Inserting (\ref{r1}) into equation (\ref{distr}) and taking into account that we are dealing with 
a classical solution of
$$
\displaystyle {\frac{\partial u}{\partial t}}= \frac{\partial}{\partial x}
\left(u\Phi\left(u^{m-1}{\frac{\partial u}{\partial x}}\right) -a(\mu-\bar{\mu})u \right)
$$
out of the singular set $ \Sigma $, we get 
$$
\sum_{ \{ +,-\} }\pm  \int _0 ^T \left[\Phi\left(u_\pm^{m-1}(t)\frac{\partial u_\pm}{\partial x}(t)\right)
+ a (\mu_\pm(t) -\bar{\mu}(t)) - \sigma'_\pm (t)\right]u_\pm (t) \psi_\pm (t)dt=0,
 $$
fo every $\psi \in C^1_{0}(]0,T[ \times \RR )$, where  $u_\pm(t)=u(t,\sigma_\pm(t)^{\mp})$,
$\frac{\partial u_\pm}{\partial x}(t)=\frac{\partial u}{\partial x}(t,\sigma_\pm(t)^{\mp})$,
$\mu_\pm(t)=\mu(t,\sigma_\pm(t))$, and $\psi_\pm (t)=\psi (t,\sigma_\pm(t))$. 

Now, using that $u(t,\sigma_-(t)^{+})=v(t,0)>0$, $u(t,\sigma_+(t)^{-})=v(t,M)>0$ and choosing test functions $\psi$  supported around one  of the branches of $\Sigma$, we find
$$
\Phi\left(u_\pm^{m-1}(t)\frac{\partial u_\pm}{\partial x}(t)\right)
+ a (\mu_\pm(t) -\bar{\mu}(t)) - \sigma'_\pm (t)=0, \quad  \mbox { a.e. } t\in ]0,T[.
 $$
Combining (\ref{nuii}), $\mu(t,\sigma_-(t))=0$,  $\mu(t,\sigma_+(t))=M$ 
and (\ref{satura}),  yields
$$
\begin{array}{l}
 c-a\bar{\mu}(t)-\sigma^\prime_-(t)=0,\\
 -c+a(M-\bar{\mu}(t))-\sigma'_+(t)=0,
 \end{array}
 $$
which gives the motivation for (\ref{desplaz}).

 Once we know $\sigma^+$, $\sigma^-$ and $\varphi$, we define the solution $u(t,x)$ in the following way: solve $x=\varphi(t,\eta)$ and obtain 
$\varphi^{-1}_t : [\sigma_-(t),\sigma_+(t)]\to [0,M]$, then define
$$
u(t,x)=\left \{ \begin{array}{ll}
               \displaystyle  \frac{1}{v(t,\varphi^{-1}_t(x))}, & \quad \mbox{ if } x\in [\sigma_-(t),\sigma_+(t)], \\
                0,    & \quad  \mbox{ otherwise.}
                \end{array}\right .
$$
This can be done because $\frac{\partial \varphi }{\partial \eta}(t,\eta)=v(t,\eta)>0$. Now, in order to prove 
 (\ref{integral}) we perform the change  of variables $x=\varphi(t,\eta)$ and obtain 
\begin{eqnarray*}
 \int_0^T \int _{\R}\left | \frac{\partial u}{\partial x}(t,x) \right | dxdt &=&
\int_0^T \int _{\sigma_-(t)}^{\sigma_+(t)}
\left | \frac{\partial u}{\partial x}(t,x) \right | dxdt \\ &=&
\int_0^T \int _{0}^{M}
\left | \frac{\partial v}{\partial \eta}(t,\eta) \right |\frac{1}{v^2(t,\eta)} d\eta dt<\infty,
\end{eqnarray*}
where we have used (\ref{distribucionale0}) to prove that last integral is finite. 

We have $u\in C^0  \left (\overline{ Q_{T}}\setminus \Sigma \right )$, but  studying 
the regularity properties of the solutions requires that $\varphi=\varphi(t,\eta)\in C^2(]0,T[\times ]0,M[)$. First, let us observe that 
$$
\frac{\partial}{\partial t} \int _0^{\eta } v(t,\delta )d \delta = \Phi \left( \frac { v_{ \eta } (t,\eta) } {v^{2+m}(t,\eta)} \right)-c-a \eta.
$$
Therefore, $\varphi_t\in C^0 ((0,T)\times [0,M])$, and  
\begin{eqnarray*}
\frac{d}{d t} \int _0^{M } \eta  v(t, \eta )d \eta
\label{dermu}
\end{eqnarray*}
is well--defined, which can be deduced from the following computation
\begin{eqnarray*}
 \int _0^{M } \eta  v_t(t,\eta )d \eta&=&\int _0^M \eta \left (\Phi \left ( \frac { v_{ \eta } (t,\eta) } {v^{2+m}(t,\eta)} \right ) \right )_{\eta}d \eta
- a \frac{M^2}{2}\\ &=&Mc- \int _0^M \eta \, \Phi \left ( \frac { v_{ \eta } (t,\eta) } {v^{2+m}(t,\eta)} \right )d \eta- a \frac{M^2}{2} .
\end{eqnarray*}
Then, $t\to \overline{\mu}(t)$ given by \eqref{muu} is a $C^1$ function. This completes the regularity of $\varphi $ because the other derivatives can be computed in the same way. 
This implies that $u\in C^0  \left (\overline{ Q_{T}}\setminus \Sigma \right )$, which is the key tool to verify (\ref{distr}).
\end{proof}
 
  \subsection{Dynamics of the support and possible concentrations}{\label{ss}}
To understand the dynamics of the support and the possible concentration towards a  Dirac mass we are going to make some precisions. Denote $\ell (t)= \sigma_+(t)-\sigma_-(t)$ the  length of the support, where $\ell(0) = \xi_+-\xi_-$. Then 
 $\ell'(t)=2c-aM$,  which raises various possibilities. In the case $M<\frac{2c}{a}$ the solution spreads out and expands and we can think in behaviors of the type $\int_{J}u(t,x) dx\to 0$ as $t\to +\infty$ for every compact interval $J$, or the convergence towards a continuous traveling wave as those studied in the next Section. We can not rule out other possibilities:   we know that the solution is bounded before $T^*$, but not in a possible continuation in time of this solution; neither can we elucidate about the eventual  rupture of the geometry of the solution giving rise to a different phenomenology. In the case $M>\frac{2c}{a}$ the solution concentrates on smaller and smaller supports and it can be conjectured that in $T^*=\frac{\ell(0)}{aM-2c}$ (formal solution of  $\ell (t)=0$) 
the solution is concentrated in a Dirac delta. It is possible that, when the solution becomes zero at the ends of its support, this concentration process leading to the Dirac mass will slow down. Finally, the case $M=\frac{2c}{a}$ the support of the solution has a  fixed length, but  in a continuation in time of the solution, none of the previous possibilities can be discarded.

It is also possible a lateral movement of the support. To see this, we consider the evolution of the central point of the support interval
$$
\sigma_c(t)=\frac{\sigma_+(t)+\sigma_-(t)}{2}.
$$
Using the formulas above, it follows that
$$
\sigma'_c(t)=a(2\bar{\mu}(t)-M),
$$ 
which is related to the distribution of the mass in the interval that defines the support of $u$. To clarify this, let us define the center of mass associated with $u$
$$
\displaystyle {m}(t)=  \frac 1 {M}\int_{\sigma_-(t)}^{\sigma_+(t)} s\,u(t,s)\, ds.
$$
Then, after an integration by parts we find
$$
M {m}(t)=M\sigma_+(t)-(\sigma_+(t)-\sigma_-(t))\bar{\mu}(t),
$$
from which we deduce 
$$
\sigma'_c(t) =\frac{aM}{2(\sigma_+(t)-\sigma_-(t))}\left ( \sigma_c(t)- {m}(t) \right ).
$$
This implies that if the mass is concentrated at one end of its support, then the solution will move to the other one.

\section{Traveling waves}

In this section, we look for solutions of (\ref{s}) in the form  $u(t,x)=\upsilon(x-\sigma t)$, being $\upsilon: \R \to \R$ a  function
with compact support  $[\xi_-, \xi_+]$, such that  $\upsilon(\xi)>0, \; \xi\in [\xi_-, \xi_+]$. In this case,
the support of $u (t,\cdot)$, given by $[\sigma_-(t), \sigma_+(t)]$, satisfies
$$
\sigma_-(t)=\xi_-+\sigma t, \; \sigma_+(t)=\xi_++\sigma t,
$$
and the function
 $$\displaystyle \mu (t,x)=\int_{-\infty}^x u(t,s)\, ds =\int_{\xi_-+\sigma t}^x \upsilon (s-\sigma t)\, ds =\varkappa(x-\sigma t),$$ 
where  $\varkappa:\R \to \R$ is definded as $\varkappa(\xi)=\int_{\xi_-}^\xi \upsilon(s)\, ds$.

It is easy to check that $\overline{\mu}(t)=\bar{\varkappa} +\sigma t$, with 
 $$
 \bar{\varkappa}= \frac 1 {\xi_+-\xi_-} \left (\xi^+ M-  \int_{\xi_-}^{\xi_+} s\upsilon(s)\, ds\right ),
 $$
with $M=\int_{\xi_-}^{\xi_+}\upsilon(s)\, ds$. Inserting into the first equation of (\ref{s}) the profile of the traveling wave, we find
\begin{equation}\label{*azul}
- \sigma \upsilon'=\left( \upsilon\Phi(\upsilon^{m-1}\upsilon')+a(\bar{\varkappa}-\varkappa)\upsilon\right)',\qquad  \xi\in (\xi_-,\xi_+),
\end{equation}
or equivalently
\begin{equation}\label{u}
\upsilon\left( \Phi(\upsilon^{m-1}\upsilon')+a(\bar{\varkappa}-\varkappa)+\sigma\right)=K, \; \mbox{ with } K\in \RR.
\end{equation}

Let us now analyze the different possibilities of the profile attending to its continuity and behavior at the ends of its support.

\medskip 

\noindent {\it Case 1: Continuous profile on the right.} We are looking for a profile fulfilling $u(\xi_+)=0$. In this framework, we have that the constant $K$ takes the value $K=0$ and  
$$
\Phi(\upsilon^{m-1}\upsilon') =-a(\bar{\varkappa}-\varkappa)-\sigma, \quad  \xi\in (\xi_-,\xi_+).
$$
Therefore, for a profile of these characteristics to exist, it is necessary that
$$
-a(\bar{\varkappa}-\varkappa)-\sigma \in [-c,c], \quad  \xi\in (\xi_-,\xi_+).
$$ 
Since $[\mu(\xi_-,t), \mu(\xi_+,t)]=[0,M],$  we deduce that 
\begin{equation}\label{sigma}
a\bar{\varkappa}-c\leq \sigma\leq c+a(\bar{\varkappa}-M).
\end{equation}
From this sequence of inequalities we have
\begin{equation}\label{M}
M\leq \frac{2c}{a}.
\end{equation}

Choosing $M$  such that fits  (\ref{M}) and taking $\sigma$ verifying (\ref{sigma}), the traveling wave profile $\upsilon$ must satisfy the following equation
\begin{equation}\label{umu}
\left\{\begin{array}{l}
\upsilon^{m-1}\upsilon'=-g(a(\varkappa-\bar{\varkappa})+\sigma),\\
\varkappa'=\upsilon.
\end{array}\right.
\end{equation}
Since $u(\xi)>0, \; \xi\in (\xi_-, \xi_+)$, the function $U(\varkappa)=\upsilon(\xi(\varkappa))$ verifies
$$
U^m\frac{dU}{d\varkappa}=-g(a(\varkappa-\bar{\varkappa})+\sigma), \qquad \varkappa\in (0,M), \; U(M)=0.
$$

Integrating this equation, we obtain 
$$
\frac a {m+1} (U(\varkappa))^{m+1}= G(a(M-\bar{\varkappa})+\sigma)-G(a(\varkappa-\bar{\varkappa})+\sigma):=H(U),
$$
where $\displaystyle G(u)=\int_0^ug(s)\, ds$, as previously defined.  

The function $H(\varkappa)$ is well--defined in $[0,M]$ and verifies $H(M)=0$. In order to have $H(\varkappa)>0, \; \varkappa\in (0,M)$, the following conditions are required: 
\begin{itemize}
\item $\sigma \geq a(\bar{\varkappa}-M)$, which yields $H'(M)\leq 0$, 
\item  $\varkappa^*= \sup \{ \varkappa <M:\, H(\varkappa)=0\} \leq 0$.
\end{itemize}
It is a simple matter to check that  $\varkappa^*=2\bar{\varkappa}-M-2\sigma/a$. Therefore, $\varkappa^*\leq 0$ only if 
\begin{equation}\label{sigma2}
\sigma\geq a\left( \bar{\varkappa}- \frac M 2\right)>a\left( \bar{\varkappa}- M\right).
\end{equation}
But if (\ref{M}) is verified, then $a( \bar{\varkappa}- \frac M 2)\geq a\bar{\varkappa}-c$ holds.  We have proved the following result.
\begin{lemma}\label{ov1}
Let $M>0$ verifying (\ref{M}) and   $0<\bar{\varkappa}<M$, and  let $\sigma$ be such that 
\begin{equation}\label{sigma3}
a\left( \bar{\varkappa}- \frac M 2\right)\leq \sigma \leq c+a(\bar{\varkappa}-M).
\end{equation}
Then, the functions
\begin{equation}\label{29'}
\begin{array}{l}
\displaystyle \upsilon(\xi):= \left[\frac{m+1}{a}\left(G(a(M-\bar{\varkappa})+\sigma)-G(a(\varkappa(\xi)-\bar{\varkappa})+\sigma)\right)\right]^{1/(m+1)}, 
\\ \displaystyle \varkappa(\xi)=\int_{\xi_-}^\xi \upsilon(\tau)\, d\tau,
\end{array}
\end{equation}
which are defined in the interval $[\xi_-, \xi_+]$, are solutions of (\ref{umu}) in $(\xi_-, \xi_+)$ and  $\upsilon(\xi_+)=0, \; \varkappa(\xi_+)=M$ hold. 
\end{lemma} 
\begin{remark}
It is easy to verify that under the hypotheses of Lemma \ref{ov1}, we have $\upsilon'(\xi_+)=-\infty$. In addition, $\upsilon(\xi)>0, \, \xi \in (\xi_-,\xi_+)$, and $\upsilon(\xi_-)=0$ if and only if $\sigma=a\left( \bar{\varkappa}- \frac M 2\right)$. In this case, we find $\upsilon'(\xi_-)=+\infty$.
\end{remark}
As a consequence of this Lemma we have
\begin{theorem}
Assume (\ref{M}) and take $\sigma= a\left( \bar{\varkappa}- \frac M 2\right)$. Then, there exists $\upsilon:[\xi_-, \xi_+]\to \RR$ so that $u(t,x)=\upsilon(x-\sigma t)$
provides a continuous entropy solution of (\ref{s}).
\end{theorem}
\begin{remark} When $a\left( \bar{\varkappa}- \frac M 2\right)< \sigma \leq c+a(\bar{\varkappa}-M)$, the functions defined in (\ref{29'}) are solutions of (\ref{umu}). However, when we extend by zero the functions $\tilde{u}(t,x)=\upsilon(x-\sigma t), \; \tilde{\mu}(t, x)=\varkappa(x-\sigma t), \; t\geq 0, \, x\in[\xi_-+\sigma t, \xi_++\sigma t]$, they do not verify the saturation condition at the point $x=\xi_- +\sigma t$, that is a discontinuity point.
\end{remark}

\

\noindent {\it Case 2: Profile with jump.}  We are now looking for traveling wave profiles that satisfy $\upsilon(\xi_{\pm})=\upsilon^{\pm}>0$, with $\upsilon'(\xi_{\pm})=\mp \infty$. Inserting $\xi=\xi_\pm$ into (\ref{u}) we obtain
$$
\displaystyle \upsilon^+(-c+a(M-\bar{\varkappa})+\sigma)=K=\upsilon^- (c-a\bar{\varkappa}+\sigma).
$$
In order  that $K = 0$,  $\sigma=a\bar{\varkappa}-c=a\bar{\varkappa}+c-aM$ has to be fulfilled. Thus, this implies 
\begin{equation}\label{sigma21}
\displaystyle  M=\frac{2c} a, \; \,  \mbox{ and }\; \sigma=a\left(\bar{\varkappa}-\frac M 2 \right)=a\bar{\varkappa}-c.
\end{equation}
Reasoning as in the previous case, we find that the traveling wave profile we are looking for must verify the following system
\begin{equation}\label{umub}
\left\{\begin{array}{l}
\displaystyle  \upsilon^{m-1}\upsilon'=g(c-a\varkappa),\\
\displaystyle  \varkappa'=\upsilon,
\end{array}\right.
\end{equation}
with boundary conditions $\upsilon(\xi_\pm)=\upsilon^\pm, \; \varkappa(\xi_-)=0,\; \varkappa(\xi_+)=M.$

Then, the function $U(\varkappa)=\upsilon(\xi(\varkappa))$ must fulfill
$$
\displaystyle  U^m\frac{dU}{d\varkappa}=g(c-a\varkappa), \; \varkappa\in (0,M), \; U(0)=\upsilon^-, \; U(M)=\upsilon^+.
$$
Integrating between $0$ and $\varkappa$, we get
$$
\displaystyle  U(\varkappa)^{m+1}=(\upsilon^-)^{m+1} + \frac{m+1} a \Big[G(c)-G(c-a\varkappa)\Big], \quad \varkappa\in[0,M].
$$
Because of the symmetry of the  function $G$, we deduce that it is necessarily satisfied $\upsilon^-=\upsilon^+ :=\upsilon^{\pm}$. 
Then, we obtain the following result:
\begin{theorem}
 Let $\displaystyle M=\frac{2c} a$  and $\upsilon^{\pm}>0$. Then, the functions
$$\begin{array}{l}
\displaystyle \upsilon(\xi):= \left\{(\upsilon^{\pm})^{m+1}+\frac{m+1}{a}\Big[G(c)-G(c-a\varkappa(\xi))\Big]\right\}^{1/(m+1)}, \\ \displaystyle  \varkappa(\xi)=\int_{\xi_-}^\xi \upsilon(\tau)\, d\tau,
\end{array}
$$
defined in the interval $[\xi_-, \xi_+]$ are solutions of  (\ref{umub}) in $(\xi_-, \xi_+)$, 
and the boundary conditions  $\upsilon(\xi_-)=\upsilon(\xi_+)=\upsilon^{\pm}$ hold. 
\end{theorem}

\begin{remark}
Note that  $\upsilon'(\xi_-)=+\infty,\; \upsilon'(\xi_+)=-\infty$.
\end{remark}

\begin{remark} If $u (t,x)=\upsilon(x-\sigma t)$ is an entropic solution on the frame 
\cite{CCCSS-SUR,CCCSS-INV}, 
then its derivative must be a Radon measure on the support. If we assume that this support is finite, then it can be shown 
that it is one of those obtained previously. However, the entropy condition does not allow us in principle to determine if the
set in which the equation (\ref{*azul}) is verified is nonempty.
\end{remark}

\section*{Appendix: Distributional and entropy solutions}
Let us first define the entropy solutions context. However, as we have pointed out before the results of the paper are perfectly valid in the general context of solutions in the  distributional sense.  Let us briefly explain these concepts, see \cite{ARMAsupp,ACM2005,CCCSS-SUR}.

To connect with the framework of entropy solutions we observe that if $ (t, x) $ is a point of discontinuity
(necessarily $ x = \sigma_{\pm} (t) $), then the abstract condition (\ref{rhcond}) is written in (\ref{desplaz}).

The flux for the differential equation at a point $(t, x)$ is given by
$$
\mathcal{F}(t,x)= \left ( \Phi
\left (u^{m-1}(t,x){\frac{\partial u}{\partial x}}(t,x)\right) -a\frac{\partial S}{\partial x}(t,x)\right ) u (t,x).
$$
If $x=\sigma_+(t)$ (the other case is similar), we take into account (\ref{satura}), (\ref{r1}) to deduce
\begin{eqnarray*}
 \mathcal{F}(t,x^+)&=& 0, \qquad \mathcal{F}(t,x^-)= u(t,\sigma_+ (t)), \\
 u(t,x^+)&= &0, \qquad  \ u(t,x^-)= u(t,\sigma_+ (t)(-c+a(M-\overline{\mu}(t)),
\end{eqnarray*}
from which we conclude that $V(t)= \sigma '(t) $ is the speed of movement of the singularity. 

A notion of entropy solution is required in this framework in order to characterize solutions with jump discontinuities, which is extended here to the $d$-dimensional case  although in this paper we are only considering  one  space dimension.

Let us introduce now some useful notation. For $a <
b$ and $l \in \R$, let $T_{a,b}(r) := \max \{\min \{b,r\},a\}$, $ T_{a,b}^l =  T_{a,b}-l$.
We denote \cite{ARMAsupp,ACM2005,CCCSS-SUR}
\[
\begin{split}
\mathcal T_r & := \{ T_{a,b} \ : \ 0 < a < b \},  
\\
\mathcal{T}^+ & := \{ T_{a,b}^l \ : \ 0 < a < b ,\, l\in \R, \, T_{a,b}^l\geq 0 \}.  
\end{split}
\]
Let us denote ${\mathcal P}$ the set of Lipschitz continuous
functions $p : [0, +\infty) \rightarrow \R$ satisfying
$p^{\prime}(s) = 0$ for $s$ large enough. We write ${\mathcal
P}^+:= \{ p \in {\mathcal P} \ : \ p \geq 0 \}$. By $L^1_{w}(0,T,BV(\R^d))$ (resp. $L^1_{loc, w}(0, T,
BV(\R^d))$) we denote the space of weakly-*
measurable functions $w:[0,T] \to BV(\R^d)$. Following Dal Maso \cite{Dalmaso} we consider the  
functional
\begin{equation}
\nonumber
\begin{array}{ll}
\displaystyle {\mathcal R}_g(u) := \int_{\Omega} g(x,u(x), \nabla
u(x)) \, dx + \int_{\Omega} g^0 \left(x,
\tilde{u}(x),\frac{Du}{\vert D u \vert}(x) \right) \,  d\vert D^c
u \vert \\ 
\qquad \qquad \displaystyle + \int_{J_u} \left(\int_{u_-(x)}^{u_+(x)}
g^0(x, s, \nu_u(x)) \, ds \right)\, d \H^{d-1}(x),
\end{array}
\end{equation}
for $u \in BV(\Omega) \cap L^\infty(\Omega)$, being ${\mathcal H}^{d-1}$ the
 $(d-1)$-dimensional
Hausdorff measure in $\R^d$.
The recession
function $g^0$ of $g$ is defined by
\begin{equation*}
\label{Asimptfunct}
 g^0(x, z, \xi) = \lim_{t \to 0^+} t\, g \left(x, z, \frac{\xi}{t}
 \right).
\end{equation*}
Then, for each
$\phi\in {C}_c(\R^d)$, $\phi \geq 0$ and $T \in \mathcal
T_r$, we define the Radon measure
$g(u, DT(u))$ by
\begin{equation}
\label{FUTab}
\begin{array}{c}
\displaystyle \langle g(u, DT(u)), \phi \rangle : = {\mathcal R}_{\phi g}(T_{a,b}(u))+
\displaystyle\int_{\{u \leq a\}} \phi(x)
\left( g(u(x), 0) - g(a, 0)\right) \, dx 
\\
 \displaystyle
\displaystyle  + \int_{\{u \geq b\}} \phi(x)
\left(g(u(x), 0) - g(b, 0) \right) \, dx,
\end{array}
\end{equation}
 where $DT$  denotes the gradient of $T$ in the sense of distributions.
 Let us introduce $h: \R \times \R^d \times \R^d\rightarrow \R$ defined by
\begin{equation}
\label{hdef}
h(z,\xi, S_x):=  {\bf F}(z, \xi, S_x) \xi, \quad \mbox{ where } \quad {\bf F}(z, \xi, S_x)=z\Phi(z^{m-1}\xi) -a S_x z ,
\end{equation}
where $S_x$ is given by \eqref{r1} and $S(t, \cdot) \in W^{2,p}$. We define $h_R(u,DT(u))$ as the Radon
measure given by (\ref{FUTab}) with $h_{R}(z,\xi) = R(z) h(z,\xi, S_x)$. Finally,  let $J_q(r)$ denote the primitive of $q$ for any function $q$; i.e. 
$$\displaystyle J_q(r):=\int_0^r q(s)\,ds.$$

\begin{definition} 
\label{DefESevP} 
{\rm
A measurable function $u\!:\,(0,T)\times \R^d \rightarrow \R^+$ is an {\it entropy solution}
of (\ref{s}) in $Q_T = (0,T)\times \R^d$ if 
\begin{itemize}
\item $u \in C([0, T];
L^1(\R^d))$

\item $u(0,x) = u_{0}(x), x \in \R^d,$

\item $T(u(\cdot)) \in L^1_{loc, w}(0, T, BV(\R^d)) \ \mbox{for all}\ T \in \mathcal
T_r$

\end{itemize}
and the equation is satisfied in the following sense:
\begin{itemize}

\item[(i)] Equation \eqref{s} is verified  in $\mathcal{D}'((0,T)\times \R^d)$

\item[(ii)] Given any truncations $R \in {\mathcal P}^+, \ T \in {\mathcal T}^+$  and any $\eta \in \mathcal{D} (Q_T)$,
 the following entropy inequality is satisfied:
\begin{eqnarray}
\label{putataghere}
\displaystyle
 \int_{Q_T} && \eta
h_R(u,DT(u)) \, dt \ +  \displaystyle \int_{Q_T} \eta h_T(u,DR(u)) \, dt
\\ \nonumber
&\leq & \displaystyle \int_{Q_T}J_{TR}(u) \partial_t\eta\ dxdt \ -  \int_{Q_T}  {\bf F}(u, \nabla u) \nabla \eta \ T(u)
R(u) \ dxdt.
\end{eqnarray}
\end{itemize}
 }
\end{definition}

The following result indicates the type of entropic solution we can expect and its proof would require to follow the steps in the proofs of \cite{ACMSV,ARMAsupp,ACM2005,CCCSS-SUR}. Here, we omit this proof for being outside the objectives of the paper and will be the content of a forthcoming publication.

\begin{theorem}
\label{EUTEparabolic}
Under the  assumptions on $\Phi$ \eqref{H1}-\eqref{H2}, for any initial datum $0 \leq u_0 \in L^1(\R^d)\cap L^\infty(\R^d)$ satisfying 
\begin{itemize}
\item[(i)] $u_0$ has jump discontinuities at the boundaries of its support;
\item[(ii)] the part of the interior of the support in which $u_0$ vanishes is a set of null measure;
\end{itemize}
there exists $T^*
> 0$ and a unique entropy solution $u$ of
(\ref{s})   in $ [0,T]\times \R^d$ , with $T<T^*$, such that $u(0) = u_0$.
Moreover, if $u(t)$,
$\overline{u}(t)$ are the entropy solutions associated with
initial data $u_0$, $\overline{u}_0 \in L^{1}(\R^d)^+$, respectively, then
\begin{equation}
\label{CPentropys}
\Vert (u(t) - \overline{u}(t))^+ \Vert_1 \leq
\Vert (u_0 - \overline{u}_0)^+ \Vert_1 \ \ \ \ \ \ {\rm for \
all} \ \   0\leq t < T^*.
\end{equation}
The value of $T^*$  until which the solution is uniformly bounded in $ [0,T]\times \R^d$ , with $T<T^*$,  is  limited by $\frac{\ell(0)}{aM-2c}$
in the case of  possible blow--up and a solution with two fronts at the ends of its support, or might be non--limited in the absence of  blow--up, in the sense given in Section \ref{ss}.
\end{theorem}


\begin{thebibliography}{[99]}

\bibitem{[AQ08]} A.R.A. Anderson,  V. Quaranta, \textit{Integrative
mathematical oncology}, Nature Reviews  Cancer {\bf 8}
(2008), 227-234.

\bibitem{ACMSV} F. Andreu, V. Caselles, J.M. Maz\'on, J. Soler, M. Verbeni, \textit{Radially symmetric solutions of a tempered diffusion equation. A porous media flux-limited case},  SIAM J. Math. Anal. {\bf 44} (2012), 1019--1049.

\bibitem{ARMAsupp}
{F.~Andreu, V.~Caselles, J.M. Maz\'on, S. Moll,} \textit{Finite Propagation Speed for Limited Flux Diffusion Equations}. Arch. Rat. Mech. Anal. {\bf 182} (2006),  269--297.

\bibitem{ACM2005}
F.~Andreu, V.~Caselles,  J.M. Maz\'on,
\newblock {\it The Cauchy Problem for a Strongly Degenerate Quasilinear
Equation},
\newblock  J. Eur. Math. Soc. {\bf 7} (2005),
361--393.


\bibitem{[BBNS07]} N. Bellomo, A. Bellouquid, J. Nieto,  J. Soler,
\textit{Multicellular growing systems: Hyperbolic limits towards
macroscopic description},  {Math. Mod. Meth. Appl. Sci.}
{\bf 17} (2007),  1675--1693.


\bibitem{[BBNS10]} N. Bellomo, A. Bellouquid, J. Nieto,  J. Soler,  \textit{Complexity and  mathematical tools toward the modelling of multicellular  growing systems},  {Math. Comput. Modelling}, \textbf{51} (2010), 441--451.

\bibitem{[BBNS10B]} N. Bellomo, A. Bellouquid, J. Nieto, J. Soler,  \textit{Multiscale biological tissue models
and flux-limited chemotaxis from binary mixtures of multicellular
growing systems}, {Math. Mod. Meth. Appl. Sci.} {\bf 10} (2010), 1179--1207.

\bibitem{BBTW} N. Bellomo, A. Bellouquid, Y. Tao, M. Winkler, \textit{Toward a mathematical theory of Keller-- Segel models of pattern formation in biological tissues}, Math. Mod. Meth. Appl. Sci. {\bf 25} (2015), 1663--1763.

\bibitem{BW2} N. Bellomo, M. Winkler, \textit{Finite-time blow-up in a degenerate chemotaxis system with flux limitation}, Trans. Amer. Math. Soc. Ser. B {\bf 4} (2017), 31--67.

\bibitem{BW1} N. Bellomo, M. Winkler, \textit{A degenerate chemotaxis system with flux limitation. Maximally extended solutions and absence of gradient blow-up}, Comm. Part. Diff. Eq. {\bf 42} (2017), 436--473.


\bibitem{B} A. Blanchet, J. Dolbeault,  B. Perthame,
\textit{Two-dimensional Keller--Segel model:
optimal critical mass and qualitative properties of the solutions}.
Electron.  J.  Diff.
Eq. {\bf 32} (2006), 44.

\bibitem{BC} N. Bournaveas, V. Calvez, \textit{The one-dimensional Keller--Segel model with fractional diffusion of cells},  Nonlinearity {\bf 23} (2010), 923. 


\bibitem{CMSV} J. Calvo, J. Maz\'on, J. Soler, M. Verbeni, \textit{Qualitative properties of the solutions of a nonlinear flux-limited equation arising in the transport of morphogens}, Math. Mod. Meth. Appl. Sci. {\bf 21} (2011), 893--937.

\bibitem{CCCSS-SUR}
J. Calvo, J. Campos, V. Caselles, O. S\'anchez, J. Soler, \emph{Flux--saturated porous media  equation and applications}. EMS Surveys in Mathematical Sciences {\bf 2(1)} (2015), 131--218.
\bibitem{CCCSS-INV}
 J. Calvo, J. Campos, V. Caselles, O. S\'anchez, J. Soler, \emph{Pattern Formation in a flux limited reaction-diffusion equation of porous media type}. Inv. Math. {\bf 206} (2016), 57--108.

\bibitem{CCCSS-EMS}
 J. Calvo, J. Campos ,  V. Caselles, O. S\'{a}nchez,  J. Soler, {\em Qualitative behaviour for flux--saturated
mechanisms: traveling waves, waiting time and
smoothing effects}, J. Eur. Math. Soc. {\bf 19} (2017), 441--472.

\bibitem{Campos}
 J. Campos ,  P. Guerrero, O. S\'{a}nchez,  J. Soler,
 \emph{On the analysis of travelling waves to a nonlinear flux limited
reaction-diffusion equation}.  Ann. Inst. H. Poincar\'e. Anal. Non Lin\'eaire {\bf 30} (2013), 141--155.

\bibitem{Campos2016}
 J. Campos, J. Soler, \emph{Qulitative behavior and traveling waves for flux-saturated porous media equations arising in optimal mass transportation}. Nonlinear Anal. TMA {\bf 137} (2016), 266--290. 

\bibitem{Carrillo}
J.A. Carrillo, V. Caselles, S. Moll, \textit{On the relativistic heat equation in one space dimension}.  Proc. London Math. Soc. {\bf 107} (2013), 1395--1423.


 \bibitem{preprintVicent}
 V. Caselles, \textit{Flux limited generalized porous media diffusion equations}, Publ. Mat. {\bf 57} (2013), 144--217.
 
  \bibitem{CCC2}
 V. Caselles, \textit{Convergence of flux limited porous media diffusion equations to its classical counterpart}. Ann. Sc. Norm. Super. Pisa Cl. Sci. {\bf (5) Vol. XIV} (2015), 481--505.

\bibitem{[CMPS04]} F.A. Chalub, P. Markovich, B. Perthame, C. Schmeiser,
 \textit{Kinetic models for chemotaxis and their drift--diffusion limits,}
Monatsh. Math. {\bf 142(1-2)} (2004), 123--141.

\bibitem{[CDMOSS06]} F.A. Chalub, Y. Dolak-Struss, P. Markowich, D. Oeltz, C. Schmeiser,  A. Soref,  \textit{Model hierarchies for cell aggregation by
chemotaxis,}  \textit{Math. Models Methods Appl. Sci.} {\bf 16}
(2006), 1173--1198.

\bibitem{[CH2]}  P.H. Chavanis,  \textit{Nonlinear mean field Fokker-Planck equations. Application to
the chemotaxis of biological populations,}
European Physical Journal B--Condensed Matter {\bf 62(2)} (2008),  179--208.

\bibitem{cher}  A. Chertock, A. Kurganov, X. Wang, Y. Wu,  \textit{On a chemotaxis model with saturated chemotactic flux,} Kinet. Relat. Models {\bf 5} (2012),  51--95

\bibitem{CKR}
{ A. Chertock, A. Kurganov, P. Rosenau},
 {\it Formation of discontinuities in flux-saturated degenerate parabolic equations},
 Nonlinearity  {\bf 16} (2003), 1875--1898.
 
 \bibitem{Dalmaso} G. Dal Maso, {\it Integral representatios on $BV(\Omega)$ of $\Gamma$-limits of variational integrals},
 Manuscripta Math. {\bf 30} (1980), 387--416.
 
\bibitem{[DYH07]} E. Dessaud, L.L. Yang, K. Hill, B. Cox, F. Ulloa, A. Ribeiro,
A. Mynett, B.G. Novitch, J. Briscoe,  \textit{Interpretation of the
sonic hedgehog morphogen gradient by a temporal adaptation
mechanism,} {Nature} {\bf 450} (2007) 717--720.

\bibitem{DP} J. Dolbeault, B. Perthame,  \textit{Optimal critical mass in the two--dimensional Keller-Segel model in
$\R^2$, }
C. R. Math. Acad. Sci. Paris {\bf 339} (2004), 611--616.

\bibitem{[DS09]} J. Dolbeault, C. Schmeiser,  \textit{The two--dimensional Keller-Segel model after blow-up,} {Disc. Cont. Dyn.
Syst} {\bf 25} (2009), 109--121.


\bibitem{[FK02]} J. Folkman, K. Kerbel,  \textit{Clinical translation of angiogenesis
inhibitors,} {Nature Review Cancer} {\bf 2} (2002),
727--739.

\bibitem{Giacomelli}
L. Giacomelli, \textit{Finite speed of propagation and waiting-time phenomena for degenerate parabolic equations with linear growth Lagrangian}, SIAM J. Math. Anal. {\bf 47 (3)} (2015), 2426--2441.


\bibitem{Gurtin1984} E.M. Gurtin, R.C. MacCamy, E.A. Socolovsky, {\textit A coordinate transformation for the porous media equation that renders the free boundary stationary}, Quart. Appl. Math. {\bf 42} (1984), 345--357. 

\bibitem{[HW00]} D. Hanahan,  R.A. Weinberg,  \textit{The hallmarks of
cancer,} Cell {\bf 100} (2000),  57--70.

\bibitem{H} T. Hillen, A. Potapov,
 \textit{The one-dimensional chemotaxis model: global existence and asymptotic profile}, Math. Mod. Meth. Appl. Sci. {\bf 27} (2004), 1783--1801.

\bibitem{[HP]} T. Hillen, K.J. Painter, \textit{A user's guide to PDE models for chemotaxis,}  {J. Math. Biol.}  {\bf 58} (2009), 183--217

\bibitem{[KS71]} E.F. Keller, L.A. Segel,  \textit{Traveling Bands of chemotactic
Bacteria: A Theoretical Analysis}, J. Theor. Biol. {\bf
30} (1971), 235--248.

\bibitem{[K80]} E.F. Keller,  \textit{Assessing the Keller-Segel model: how has it
fared? In biological growth and spread}, Proc. Conf., Heidelberg,
1979, 379--387, Springer Berlin, 1980.

\bibitem{KM} S. Kondo, T. Miura,  \textit{Reaction-Diffusion Model as a Framework for Understanding Biological Pattern Formation}, { Science} {\bf 329} (2010), 1616--1620.

\bibitem{Ros-Non} A. Kurganov, P. Rosenau, \textit{On reaction processes with saturating diffusion},
Nonlinearity {\bf 19}  (2006), 171--193 .

\bibitem{Li} G.M. Lieberman,  \textit{Second order parabolic differential equations}, World Scientific, 2005.

\bibitem{[LAC]} M. Lachowicz,  \textit{Micro and meso scales of description
corresponding to a model of tissue invasion by solid tumors,}
Math. Mod. Meth. Appl. Sci. {\bf 15} (2005), 1667--1683.

\bibitem{Ladyparabolico}
O. A. Ladyzhenskaja,  V.A. Solonnikov, N.N. Ural'ceva, \emph{Linear and quasi-linear equations of parabolic type}, AMS-Translations of Mathematical Monographs, 23, Providence, RI, 1968.

\bibitem{Levermore81}
C.D.~Levermore, G.C.~Pomraning, {\it A Flux-Limited Diffusion Theory}, { Astrophys. J.} {\bf 248} (1981), 321--334.

\bibitem{MePuSh97} A.M. Meirmanov, V.V. Pukhnachov, S.I. Shmarev, \newblock{\it Evolution equations and Lagrangian coordinates}, De Gruyter Expositions in Mathematics, 24, 1997.

\bibitem{MM}
D. Mihalas, B. Mihalas,
\newblock {\it Foundations of radiation hydrodynamics},
\newblock Oxford University Press, 1984.

\bibitem{[OH02]} H. Othmer  and T. Hillen,  \textit{The diffusion limit of
transport equations II: chemotaxis equations,} {SIAM J.
Appl. Math.} {\bf 62} (2002), 1222--1250.

\bibitem{[PA]} C.S. Patlak,  \textit{Random walk with persistant and external bias,}
{Bull. Math. Biophys.} {\bf 15} (1953), 311--338.

 \bibitem{Rosenau2}
{ P. Rosenau},
\textit{Tempered Diffusion: A Transport Process with Propagating Front and Inertial Delay},
 Phys. Review A {\bf 46} (1992), 7371--7374.

\bibitem{VGRaS} M. Verbeni, O. S\'anchez, E. Mollica, I. Siegl-Cachedenier, A. Carleton, I. Guerrero, A. Ruiz i Altaba, J. Soler, \textit{Modeling morphogenetic action  through flux-limited spreading},  Phys. Life Rev. {\bf 10} (2013), 457--475.

\end{thebibliography}
\end{document}